\documentclass[11pt]{amsart}

\usepackage{amsmath, graphicx, cite, enumerate, comment}
\usepackage{amssymb, amsthm, amscd}
\usepackage{latexsym}
\usepackage{amssymb}
\usepackage[english]{babel}
\usepackage{amsmath,amssymb}
\usepackage{amsfonts, dsfont}
\usepackage{amscd}
\usepackage{cite}
\usepackage{marvosym}
\usepackage{mathrsfs}
\usepackage{amsthm}
\usepackage{hyperref}
\usepackage{verbatim}
\usepackage{setspace}

\usepackage{tikz} \usetikzlibrary{arrows,shapes,patterns} % AC: You need this command for everything below to work
\usepackage{lipsum} % for sample

\usepackage{tikz-cd}
\usetikzlibrary{arrows, arrows.meta}
\usepackage{dsfont}
\usetikzlibrary{matrix}
\tikzcdset{arrow style=tikz, diagrams={>=stealth}}
\usepackage{mathtools}

\newtheorem{thm}{Theorem}[section]
\newtheorem{lem}[thm]{Lemma}
\newtheorem{cor}[thm]{Corollary}
\newtheorem{prop}[thm]{Proposition}

\theoremstyle{remark}
\newtheorem{rmk}[thm]{Remark}

\numberwithin{equation}{section}

\newcommand{\R}{\mathbb{R}}
\newcommand{\p}{\partial}

\newcommand{\Diff}{\textup{Diff}}
\newcommand{\Aut}{\textup{Aut}}

\newcommand{\be}{\begin{equation}}
\newcommand{\ee}{\end{equation}}
\newcommand{\bd}{\begin{displaymath}}
\newcommand{\ed}{\end{displaymath}}

     \title[Contractibility results for certain spaces of Riemannian metrics on the disc]{Contractibility results for certain spaces of Riemannian metrics on the disc}
     \author{Alessandro Carlotto and Damin Wu}
     \address{ \noindent Alessandro Carlotto: 
     	\newline ETH D-Math, R\"amistrasse 101, 8092 Z\"urich, Switzerland 
     	\newline IAS, 1 Einstein drive, 08540 Princeton, United States of America
     	\newline
     	 \textit{E-mail address: alessandro.carlotto@math.ethz.ch, alessandro.carlotto@ias.edu} 
     	 \newline \newline \indent Damin Wu: 
     	 \newline University of Connecticut - Department of Mathematics, 341 Mansfield Road U 1009, 06269 Connecticut, United States of America
     	 \newline \textit{E-mail address: damin.wu@uconn.edu} }
     
\begin{document}
     	
     	\begin{abstract}
     	We provide a general contractibility criterion for subsets of Riemannian metrics on the disc. For instance, this result applies to the space of metrics that have positive Gauss curvature and make the boundary circle convex (or geodesic). The same conclusion is not known in any dimension $n\geq 3$, and (by analogy with the closed case) is actually expected to be false for many values of $n\geq 4$. 
     	\end{abstract}
     
      	\maketitle

   	\begin{spacing}{1.04}
							
	\section{Introduction} \label{sec:intro}
	
	Let $M$ be a smooth, compact manifold of dimension $n\geq 2$, without boundary.
	Two important questions in Riemannian Geometry are whether $M$ supports any metric of positive scalar curvature and, if that is the case, whether one can characterize the homotopy type of the subset $\mathfrak{R}^{+}(M)\subset \mathfrak{R}(M)$ consisting of all such metrics. In spite of many remarkable advances, these questions are still far from being fully resolved, except for the case of surfaces ($n=2$) and of 3-manifolds ($n=3$), the latter being the object of the very recent work by Bamler-Kleiner \cite{BK19}.
	
	When $n\geq 3$, it is well-known that \emph{any} manifold $M$ as above can be endowed with metrics of negative scalar curvature (cf. \cite{KW75a, KW75b} for a more precise characterization), while pioneering works by Schoen-Yau \cite{SY79a, SY79b,SY82} and Gromov-Lawson \cite{GL80a,GL83} imply that, at the opposite end of the spectrum, manifolds that do not admit any metric of positive scalar curvature exist in abundance. The simplest such example is provided by the three-dimensional torus $S^1\times S^1\times S^1$, which is in fact a significant result as it implies the positive mass theorem for three-dimensional asymptotically flat manifolds (this is the compactification approach followed in \cite{SY17}). A complete description of those three-manifolds for which $\mathfrak{R}^{+}(M)\neq\emptyset$ follows, when $n=3$, by Perelman's papers on the Ricci flow with surgery \cite{Per02, Per03a, Per03b}, while for $n\geq 5$ and $M$ simply connected our knowledge relies on the combination of outstanding work by Gromov-Lawson \cite{GL80b} and Stolz \cite{Sto92}: $\mathfrak{R}^{+}(M)\neq\emptyset$ if and only if either $M$ does not admit any spin structure or, in case of spin manifolds, if $\alpha(M)=0$. However, for $n\geq 4$, we are far from a complete understanding of the topology of $\mathfrak{R}^{+}(M)$ (for partial, but often striking achievements see e.g. \cite{Hit74, Car88, KS93, BG96, Rub01, BHSW10, HSS14, BER17} among others). These works provide simple criteria (of dimensional character, and/or on the topology of $M$) implying that $\mathfrak{R}^{+}(M)$ is often \emph{not even path-connected}, which contrasts with the main theorem in \cite{Mar12} for $S^3$ (and, more generally, with the results in \cite{BK19}).
	
	For closed surfaces, the Gauss-Bonnet theorem gives at once that $\mathfrak{R}^{+}(M)\neq\emptyset$ if and only if $\chi(M)>0$, that is if $M \cong S^2$ or $\chi\cong \mathbb{R}\mathbb{P}^2$. The topology of this space of metrics was first studied by Weyl \cite{Wey15} to the scope of proving the existence of isometric embeddings of positively curved spheres as convex bodies in $\R^3$ (a theorem later obtained by Nirenberg \cite{Nir53} in one of his very first papers). In particular, Weyl was able to prove that $\mathfrak{R}^{+}(S^2)$ is path-connected. To the best of our knowledge, the full characterization of the homotopy type of $\mathfrak{R}^{+}(S^2)$ and $\mathfrak{R}^{+}(\mathbb{R}\mathbb{P}^2)$, namely the theorem that these spaces are actually \emph{contractible}, was first sketched by Rosenberg-Stolz in their beautiful survey \cite{RS01}.

	 If we now let $M$ be a smooth, compact manifold of dimension $n\geq 2$, with (smooth) boundary, a remarkable theorem by Gromov implies that $M$ \emph{always} supports metrics of positive scalar (in fact: sectional) curvature \cite{Gr69}. Hence, \emph{boundary conditions} must be introduced for the problem not to be trivial. This leaves room for different sorts of choices, and we will focus on some of the most natural ones, namely those defined by pointwise conditions given in terms of the mean curvature of the boundary.
	 
	 For instance, one can consider the space $\mathfrak{M}^{+}(M)$ that consists of those Riemannian metrics on $M$ that are smooth up to the boundary, and such that $(M,g)$ is a (strictly) mean-convex domain of positive scalar curvature. 
	 Once again, one wonders under what conditions on $M$ it is the case that  $\mathfrak{M}^{+}(M)\neq\emptyset$ and what is the topology of this space of metrics. Yet, some of the techniques that come into play in the study of closed manifolds do not have a straightforward extension to the case when $\partial M\neq\emptyset$, with the net result that comparatively little is known. Besides the analysis of the $n=3$ case, which is the object of the recent article \cite{CL19} by the first author and C. Li, we mention the work by Walsh \cite{Wal14} (although the focus there is on \emph{collar} boundary conditions) and Botvinnik-Kazaras \cite{BK18}. The main purpose of this note is to prove a general theorem for the (two-dimensional) closed disc $\overline{D}$, which easily provides various significant geometric applications, including a full answer to the question about the homotopy type of
	$\mathfrak{M}^{+}(M)\neq\emptyset$ in the special case above.

	 In order to state our main result, let us first introduce some notation. Given $g\in\mathfrak{R}(M)$, thus in the space of Riemannian metrics on $M$ \emph{that are smooth up to the boundary}, denote by $[g]$ the pointwise conformal class of $g$; that is, $g_1 \in [g]$ if $g_1 = e^{2u} g$ for some smooth function $u$ on $M$. Let
	 \[
	 \mathfrak{C}(M) = \{ [ g ] : g \in \mathfrak{R}(M)\},
	 \]
	 which we always assume endowed with the so-called \emph{smooth topology} inherited (as a quotient) from $\mathfrak{R}(M)$; denote by $\pi: \mathfrak{R}(M)\to \mathfrak{C}(M)$ the associated projection map.

	 \begin{thm}\label{thm:Main}
	 	Let $\mathfrak{M}^+(\overline{D})\subset \mathfrak{R}(\overline{D})$ be a non-empty space of Riemannian metrics satisfying the following two properties:
	 	\begin{enumerate}
	 		\item {for every $g\in \mathfrak{R}(\overline{D})$ the intersection $\mathfrak{M}^+(\overline{D}) \cap \pi^{-1}([g])$ is convex in the sense that if $g_i = e^{2u_i} g \in \mathfrak{M}^+(\overline{D}) \cap \pi^{-1} ([g])$, $i = 1, 2$, then
	 			\[
	 			e^{ 2 ( t_1 u_1 + t_2 u_2 )} g \in \mathfrak{M}^+ (\overline{D}) \cap \pi^{-1} [g] \ \text{for all} \ t_1, t_2\geq 0 \ \text{such that} \ t_1+t_2=1;
	 			\]
	 			}
	 		\item {$\mathfrak{M}^+(\overline{D})$ is invariant under diffeomorphisms, i.e. 
	 			\[
	 			g\in \mathfrak{M}^+(\overline{D}) \ \text{if and only if} \ \phi^{\ast}g\in\mathfrak{M}^+(\overline{D}) \ \text{for all} \ \phi\in \textup{Diff}(M).
	 			\] 
	 		}
	 		\end{enumerate}
 		Then $\mathfrak{M}^+(\overline{D})$ is contractible.
	 \end{thm}

	The two assumptions above are often easy to verify: the second one is always satisfied when $\mathfrak{M}^+(\overline{D})$ is given by \emph{curvature} conditions (which are our main concern), while the first one holds (for instance) for linear inequalities (possibly equalities) involving the Gauss curvature of the disc in question, and the geodesic curvature of its boundary. %We adopt the convention that the geodesic curvature of the flat unit disc in the complex plane is positive, and equal to one. 
	Indeed, we have the well-known equations
	\[
	e^{2u} K_{e^{2u} g} = K_g - \Delta_g u, \ \ \ \ e^{u}k_{e^{2u} g}=k_g+\nu(u)
	\]
	where $\nu$ is the outward-pointing unit normal along the boundary. 
	Thus, as a special case, we can derive what follows:
	
	\begin{cor}\label{cor:R+D}
	Let $\mathfrak{M}^+(\overline{D})$ denote one of the following spaces of Riemannian metrics:
	\begin{itemize}
		\item{$I:= \left\{g\in \mathfrak{R}(\overline{D}) \ : \ K_g\geq 0\right\}, \ II:= \left\{g\in \mathfrak{R}(\overline{D}) \ : \ K_g>0\right\}, \ III:= \left\{g\in \mathfrak{R}(\overline{D}) \ : \ K_g= 0\right\}$;}
	    	\item{$I_{\partial}:= \left\{g\in \mathfrak{R}(\overline{D}) \ : \ k_g\geq 0\right\}, \ II_{\partial}:= \left\{g\in \mathfrak{R}(\overline{D}) \ : \ k_g> 0\right\}, \ III_{\partial}:= \left\{g\in \mathfrak{R}(\overline{D}) \ : \ k_g=0\right\}$.}
	\end{itemize}
 Then $\mathfrak{M}^+(\overline{D})$ is contractible; furthermore any space $\mathfrak{M}^+(\overline{D})$ obtained by intersecting one among $I, II, III$ with one among $I_{\partial}, II_{\partial}, III_{\partial}$ is either empty or contractible.
	\end{cor}

The arguments we present turn out to be rather direct, and conceptually transparent. In addition, as we will see, they are somewhat more elementary than in the case of $S^2$ (cf. Remark \ref{rem:Contract}). Note that when $n\geq 4$ the analogy with the closed case (as briefly summarized above) suggests that even $\pi_0 (\mathfrak{M}^{+}(M))$ should be non-trivial in many cases of interest. It is reasonable to expect some sort of progress on these matters in the years to come.

		 	\
		 	
		\noindent \textit{Acknowledgements:} The first author would like to thank Alexander Kupers for various clarifications; the second author would like to thank Simon Donaldson for bringing Ahlfors' work on the Beltrami equation to his attention. The authors would also like to thank the anonymous referee for providing very helpful remarks that improved the presentation of this paper.

		A. C. is partly supported by the National Science Foundation (through grant DMS 1638352) and by the Giorgio and Elena Petronio Fellowship Fund. D. W. is partly supported by the National Science Foundation (through grant DMS 1611745) and the Simonyi Endowment of the Institute for Advanced Study.
		
		This project was developed at the Institute for Advanced Study during the special year \emph{Variational Methods in Geometry}: both authors would like to acknowledge the support of the IAS and the excellent working conditions.

		\section{Proof of Theorem \ref{thm:Main}}\label{sec:proofs}

		Let $\textup{Diff}(M)$ be the group of diffeomorphisms of $M$; when $M$ is a manifold with boundary we agree that diffeomorphisms are proper in the sense that they map $\p M$ onto $\p M$. When $M$ is oriented (as we assume), we denote by $\Diff^+(M)$ the orientation-preserving subgroup of $\Diff(M)$. In either case, we also tacitly employ the (strong) smooth topology on these spaces of maps. Notice that in case $\partial M\neq\emptyset$ we do \emph{not} require diffeomorphisms to restrict to the identity map along the boundary. 
			If $M$ is a compact surface and we fix three distinct points $x_1, x_2, x_3$ on $M$, then we let
		\[
		\Diff_{\bullet}^+(M) = \{ \phi \in \Diff^+(M): \phi (x_i) = x_i, \ i = 1, 2, 3\}.
		\]
		Specifically, for $M=\overline{D}$ we agree to choose
		\[
		x_1 = 1, \ x_2 = i, \ x_3 = -1.
		\]
			As usual, we denote by $C^{\infty}(M)$ be the set of (real-valued) smooth functions on $M$. The notation $\mathds{1}_X$ stands for the identity map on a set $X$.

		Let $z = x + i y$ be the standard complex coordinate on $\mathbb{C}$, and let
		\[	g_0= |dz|^2 := (d z \otimes d\bar{z} + d\bar{z} \otimes dz)/2=  dx \otimes dx + dy \otimes dy, 
		\] the standard Euclidean metric on the plane.
		More generally, we adopt the convention that $|\vartheta|^2:= \text{Re}(\vartheta\otimes\overline{\vartheta})=(\vartheta\otimes\overline{\vartheta}+\overline{\vartheta}\otimes\vartheta)/2$ for any complex valued 1-form $\vartheta$.
		We will identify $D$ with
		the unit disc $\mathbb{D}\subset\mathbb{C}$ (as smooth manifolds) and, correspondingly, we will write $C^{\infty}(\overline{D},D)\subset C^{\infty}(\overline{D},\mathbb{C})$ to denote the set of complex-valued maps defined on $\overline{D}$ whose image lies in $\left\{z\in\mathbb{C} \ : \ |z|<1 \right\}$ and that are smooth up to the boundary. 
		
		\

		We shall employ the uniformization theorem in the following version:
		\begin{thm} \label{th:Ufz-S2}
			For any $g \in \mathfrak{R}(\overline{D})$, there exists a map $\phi \in \Diff^+(\overline{D})$ and a function $u \in C^{\infty}(\overline{D})$ such that
			\[
			g = \phi^* (e^{2u} g_0).
			\] 
		\end{thm}

		  To properly interpret the result above, and for later scopes, it is actually convenient to briefly recall the proof, recast in terms of Beltrami's equation.
		Following e.g. \cite{EE69, ES70} it is easy to construct a bijection from $\mathfrak{C}(\overline{D})$ to the space $C^{\infty}(\overline{D},D)$. 
			To define it, given $[g] \in \mathfrak{C}(\overline{D})$ with
			\[
			g = g_{11} dx\otimes dx + g_{12} (dx \otimes dy + dy \otimes dx) + g_{22} dy \otimes dy, 
			\]
			observe that
			\begin{align*}
			\ \ \ 4 g  &= (g_{11} - g_{22} - 2 i g_{12} ) dz \otimes dz + (g_{11} - g_{22} + 2 i g_{12}) d \bar{z} \otimes d \bar{z} 
			+ (g_{11} + g_{22}) (dz \otimes d\bar{z} + d \bar{z} \otimes dz) \\
		\ \ \	&= \rho | dz + \mu d \bar{z}|^2, 
			\end{align*}
			provided we set
			\[
			\rho = (g_{11}+g_{22}) + 2\sqrt{g_{11} g_{22} - g_{12}^2}, \ \
			\mu = \frac{(g_{11} - g_{22}) + 2 i g_{12}}{(g_{11} + g_{22}) + 2 \sqrt{g_{11} g_{22} - g_{12}^2}}.
			\]
			Note that
			\[
			|\mu|^2 = \frac{(g_{11} + g_{22}) - 2 \sqrt{g_{11} g_{22} - g_{12}^2}}{(g_{11} + g_{22}) + 2 \sqrt{g_{11} g_{22} - g_{12}^2}} < 1,
			\]
			hence the map $[g] \mapsto \mu$ is well-defined. 
			Conversely, given $\mu \in C^{\infty}(\overline{D}, D)$, let
			$
			g = |dz + \mu d\bar{z}|^2.
			$
			Then, the formula
			\begin{align*}
			g & = |1 + \mu|^2 dx \otimes dx + |1 - \mu|^2 d y \otimes dy +  i (\bar{\mu} - \mu) (dx \otimes dy + dy \otimes dx) 
			\end{align*}
			defines a Riemannian metric on $\overline{D}$, since $|1 \pm \mu|^2 > 0$ and
			\[
			|1 + \mu|^2 |1 - \mu|^2 - ( i (\bar{\mu}-\mu))^2 =  (1 - |\mu|^2)^2 > 0. 
			\]
			It is readily checked that the map $\mu \mapsto [g]$ is the inverse of the previous one. Furthermore, both maps in question are continuous  with respect to the smooth topologies, so that one actually obtains a homeomorphism $\mathfrak{C}(\overline{D})$ to the space $C^{\infty}(\overline{D},D)$. As a result, we can draw the following important conclusion.
			
			\begin{prop}\label{pro:ContractibleD}
				The space $\mathfrak{C}(\overline{D})$ is contractible.
			\end{prop}
		
		\begin{proof}
			The assertion follows at once, by virtue of the homeomorphism above, from the (strict) convexity of $C^{\infty}(\overline{D}, D)$.
		\end{proof}

\begin{rmk}\label{rem:Contract}
	\emph{(a)} Proposition~\ref{pro:ContractibleD} might actually be regarded as a special case of a more general fact, as we discuss in Appendix \ref{app:ConfClass} below for the sake of completeness (cf. \cite[p. 213]{Don11} and, for compact oriented manifolds, see Theorem 2.2 in \cite{FT84} applied for $s=\infty$ based on Remark 2.5 therein).\newline 
\emph{(b)} In proving the contractibility of $\mathfrak{C}(S^2)$, which is crucial for the contractibility of $\mathfrak{R}^{+}(S^2)$, Rosenberg-Stolz invoke the celebrated theorem by Smale \cite{Sma59} on the space of diffeomorphisms of the sphere.\newline
\emph{(c)} Let us explicitly remark that the sole contractibility of $\mathfrak{C}(M)$ does not suffice, in general, for the conclusion of Theorem \ref{thm:Main} to hold, as it is shown for instance by the case of $\mathfrak{R}^+(S^8)$, because of \cite{Hit74}.
\end{rmk}

		 Let $\mathbb{H} = \{z \in \mathbb{C}: \textup{Im} (z) > 0\}$ and $h: \mathbb{H} \to \mathbb{D}$ be the Cayley transform $h(z) = -(z-i) / (z+i)$. Then $h$ is biholomorphic from $\mathbb{H}$ to $\mathbb{D}$, extends smoothly to a map $\overline{\mathbb{H}}:=\left\{\text{Im}(z)\geq 0\right\}\cup\left\{\infty \right\}\to \overline{\mathbb{D}}$ and defines (by restriction) a smooth map from the boundary $\p \mathbb{H}$ onto the unit circle $\p \mathbb{D}$. Therefore, given $g\in\mathfrak{R}(\overline{D})$ the conformal class $[g]$ uniquely determines a smooth function $\mu$ on $\overline{\mathbb{H}}$ with $\sup |\mu| \le c < 1$.
		
		\
		At that stage, it follows from Ahlfors-Bers~\cite{AB60} that there is a unique map $w$ which is a diffeomorphism from $\overline{\mathbb{H}}$ onto itself, leaves $0$, $1$, $\infty$ fixed, and satisfies the  Beltrami equation $w_{\bar{z}} = \mu w_z$. This implies that $w \circ h^{-1}$ is the isothermal coordinate map for the metric $g$; hence, the map $\phi = h \circ w \circ h^{-1}$ is a biholomorphic map from $(D, g)$ onto $(D,g_0)$. Theorem \ref{th:Ufz-S2} for the closed disc follows at once. So, for $[g]\in\mathfrak{C}(\overline{D})$ we shall set $\Phi([g])=\phi$; 
		note that $\Phi([g])\in \Diff_{\bullet}^+(\overline{D})$.

\begin{prop}\label{pro:homeo}
The map
\[
\Phi: \mathfrak{C}(\overline{D})\to \Diff^+_{\bullet}(\overline{D})
\] is a homeomorphism.
\end{prop}

\begin{proof}

The continuity of the map $\Phi$ follows from the continuity of the three maps
\[
\begin{tikzcd}
\mathfrak{C}(\overline{D}) \arrow{r} & C^{\infty}(\overline{D}, \mathbb{C}) \arrow{r} & \textup{Diff}^+_{\bullet} (\overline{\mathbb{H}})  \arrow{r} & \textup{Diff}^+_{\bullet} (\overline{D}) \\
\textup{[}g\textup{]} \arrow[mapsto]{r} & \mu \arrow[mapsto]{r} & w \arrow[mapsto]{r} & \phi:=h \circ w \circ h^{-1}
\end{tikzcd}
\]

the first and third claim being obvious, the second being proven in \cite[2B]{ES70}. 

We then assert that the map in question is bijective, and that its inverse is the (patently continuous) map
$\Psi: \Diff^+_{\bullet}(\overline{D}) \to \mathfrak{C}(\overline{D})$ defined by $\phi \mapsto [\phi^*g_0]$. The fact that $\Psi\circ \Phi=\mathds{1}_{\mathfrak{C}(\overline{D})}$ descends from the biholomorphicity of $\phi=\Phi([g])$ as a map from $(D,g)$ to $(D,g_0)$. Instead, the fact that \[
\Phi\circ \Psi=\mathds{1}_{\Diff^+_{\bullet}(\overline{D})}
\] follows from observing that if $\tilde{\phi}^*(g_0)$ is (pointwise) conformal to $\phi^*(g_0)$, in the sense that $\tilde{\phi}^*(g_0) = e^{2u} \phi^*(g_0)$ for some $u \in C^{\infty}(\overline{D})$, then $\tilde{\phi} = \phi$. To see this, it is sufficient to show 
\begin{equation} \label{eq:tphi}
\tilde{\phi} \circ \phi^{-1} \in \textup{Aut}(\overline{D}),
\end{equation}
since $\tilde{\phi} \circ \phi^{-1}$ fixes three points (see for example \cite{Lan99}). Here $\textup{Aut}(\overline{D})$ denotes the set of biholomorphic automorphisms of the unit disc in $\mathbb{C}$.

In order to verify \eqref{eq:tphi}, let $\eta \in \Diff^+(\overline{D})$ satisfy
$
\eta^* g_0 = e^{2w} g_0$ for some $w \in C^{\infty}(\overline{D})$: we claim that any such map $\eta$ is holomorphic with respect to the coordinate $z$ on $\mathbb{C}$, i.e.,
\begin{equation} \label{eq:dphi=0}
\frac{\p \eta}{\p \bar{z}} = 0.
\end{equation}
(It then follows from the holomorphic version of the inverse function theorem (cf. e.g. Theorem 9.6 in \cite{Gri89}) that $\eta^{-1}$ is also holomorphic, hence $\eta \in \Aut(\overline{D})$).
To check that \eqref{eq:dphi=0} holds, first note that 
\[
\eta^* g_0 
=  \frac{\p \eta}{\p z} \frac{\p \overline{\eta}}{\p z} dz \otimes dz +  \frac{\p \eta}{\p \bar{z}} \frac{\p \overline{\eta}}{\p \bar{z}} d\bar{z} \otimes d\bar{z} 
+ \Big(\Big| \frac{\p \eta}{\p \bar{z}} \Big|^2 + \Big|\frac{\p \eta}{\p z} \Big|^2\Big) |dz|^2,
\]
so that, by our assumption, we have
\[
\frac{\p \eta}{\p z} \frac{\p \overline{\eta}}{\p z} = 0, \quad \frac{\p \eta}{\p \bar{z}} \frac{\p \overline{\eta}}{\p \bar{z}} = 0.
\]
On the other hand, since $\eta$ preserves the orientation,
\begin{align*}
0 < \textup{Jacobian} (\eta) = \frac{\p (\eta, \overline{\eta}\,)}{\p (z, \bar{z})}
= \Big| \frac{\p \eta}{\p z} \Big|^2 - \Big| \frac{\p \eta}{\p \bar{z}} \Big|^2,
\end{align*}
thus $\p \eta/\p z$ will not vanish at any point. Thereby, it follows that
$|\p \eta/\p \bar{z}|^2 = 0$,
as we had to prove. 
\end{proof}		

   Theorem~\ref{thm:Main} follows immediately from Proposition~\ref{pro:ContractibleD} and the following Lemma~\ref{le:R+drC}. 
    
    \begin{lem} \label{le:R+drC}
    	Let $\mathfrak{M}^{+}(\overline{D})$ be as in the statement of Theorem~\ref{thm:Main}. Then it is homotopy equivalent to $\mathfrak{C}(\overline{D})$.
    \end{lem}

    \begin{proof}
    	Let $g_+\in \mathfrak{M}^{+}(\overline{D})$ (the set is assumed to be non-empty). By virtue of Theorem \ref{th:Ufz-S2}, we can write
    	$g_+=\phi_+^*(e^{2u_0}g_0)$ for some $\phi_+\in \text{Diff}^+(\overline{D})$, hence by our assumption (2) on the set $\mathfrak{M}^{+}(\overline{D})$ it must be that $g_D:=e^{2u_0}g_0\in\mathfrak{M}^{+}(\overline{D})$ as well.

    	Let $\pi: \mathfrak{M}^{+}(\overline{D}) \to \mathfrak{C}(\overline{D})$ denote the restriction of the projection map, i.e. $\pi(g)=[g]$.
    	We further define the map $\sigma: \mathfrak{C}(\overline{D}) \to \mathfrak{M}^{+}(\overline{D})$ by factoring through $\Diff^+_{\bullet}(\overline{D})$ as shown in the diagram 
    	\[
    	\begin{tikzcd}
    	\mathfrak{C}(\overline{D}) \arrow{r}{\sigma} \arrow{dr}{\Phi} & \mathfrak{M}^{+}(\overline{D}) \\
    	& \Diff^+_{\bullet}(\overline{D}) \arrow{u},
    	\end{tikzcd}
    	\]
    	where $\Phi$ is the map defined above.
    	Indeed, given $[g] \in \mathfrak{C}(\overline{D})$, there exists a unique map $\Phi([g]) = \phi \in \Diff^+_{\bullet}(\overline{D})$ such that 
    	\[
    	[g] = [\phi^*g_0] = [ \Phi([g])^* g_0].
    	\]
    	Then we define
    	\[
    	\sigma ([g]) = \phi^*g_D = \Phi([g])^* g_D,
    	\]
    	where it should be noted that $\sigma ([g]) \in \mathfrak{M}^{+}(\overline{D})$ by diffeomorphism invariance (assumption (2) in Theorem \ref{thm:Main}).
    	Clearly $\pi$ is continuous; the map $\sigma$ is also continuous, as $\Phi$ is, by virtue of Proposition \ref{pro:homeo}.
    	
    	\
    	
    	Note that $\pi \circ \sigma = \mathds{1}_{\mathfrak{C}}$, for indeed
    	$
    	\pi \circ \sigma ([g]) = \pi \big((\Phi([g])^* g_D\big) = [ \Phi([g])^*g_D]=[ \Phi([g])^*g_0] = [g].
    	$
    	Next, we will show that $\sigma \circ \pi \simeq \mathds{1}_{\mathfrak{M}^{+}(\overline{D})}$, where $\simeq$ stands for the homotopy relation. Define a map $H : [0, 1] \times \mathfrak{M}^{+}(\overline{D}) \to \mathfrak{M}^+(\overline{D})$ as follows. For each metric $g$ in $\mathfrak{M}^{+}(\overline{D})$, 
    	we can write $g = e^{2u} \Phi([g])^*g_D$ for some $u \in C^{\infty}(\overline{D})$ which is uniquely determined by $g$, and continuously depending on $g$. We then set
    	\[
    	H(t, g) = e^{2(1-t)u} \Phi([g])^* g_D.
    	\]
    	Then $H$ is well-defined, by assumption (1) on the set $\mathfrak{M}^{+}(\overline{D})$, and continuous. Note that
    	\begin{align*}
    	H(0, g) & = e^{2u} \Phi([g])^*g_D = g, \\
    	H(1, g) & = \Phi([g])^*g_D = \sigma \circ \pi (g).
    	\end{align*}
    	This proves that $\sigma \circ \pi \simeq \mathds{1}_{\mathfrak{M}^{+}(\overline{D})}$. Hence, $\mathfrak{M}^{+}(\overline{D})$ is homotopy equivalent to $\mathfrak{C}(\overline{D})$. 
    \end{proof}

\begin{rmk}The conclusion of Theorem \ref{thm:Main} continues to hold if assumption (1) is relaxed to only require \emph{one} fiber to be \emph{star-shaped}. More precisely, let $\mathfrak{M}^+(\overline{D})$ satisfy condition (2) in that statement, together with the following assumption: there exists a metric $\bar{g} \in \mathfrak{M}^+(\overline{D})$ such that $e^{2tu} \bar{g} \in \mathfrak{M}^+(\overline{D}) \cap \pi^{-1}([\overline{g}])$ for all $0 \le t \le 1$, whenever $g_1 = e^{2u} \bar{g} \in \mathfrak{M}^+(\overline{D}) \cap \pi^{-1}([\bar{g}])$. Then, modifying the proof of \ref{le:R+drC} above by letting $g_+$ be $\bar{g}$ one can still conclude that $\mathfrak{M}^+(\overline{D})$ is contractible. 
	\end{rmk}

  \appendix 
    
  \section{Contractibility of the space of conformal classes} \label{app:ConfClass}

    \begin{lem} \label{le:CM-c}
    	Let $n\geq 2$ and let $M$ be an $n$-dimensional oriented manifold (possibly with non-empty boundary). Then, $\mathfrak{C}(M)$ is contractible.
    \end{lem}

    \begin{proof}
    	Fixed any $g_0\in\mathfrak{R}(M)$, define a continuous map $H: [0, 1] \times \mathfrak{R}(M) \to \mathfrak{C}(M)$ by
    	\[
    	H(t, g) = \pi \bigg( (1 - t) \bigg(\frac{\mu(g_0)}{\mu (g)}\bigg)^{2/n} g + t g_0 \bigg).
    	\]
    	Here $\mu(g)$ (resp. $\mu(g_0)$) stands for the volume element of $g$ (resp. $g_0$); in terms of local oriented coordinates $(x^1,\ldots, x^n)$ centered at an interior point of $M$
    	we have
    	\[
    	\mu(g) = \sqrt{ \det (g_{ij})} \, dx^1 \wedge \cdots \wedge dx^n.
    	\]
    	Note that the ratio
    	\[
    	\frac{\mu(g_0)}{\mu(g)} := \frac{\sqrt{\det (g_{0, ij})}}{\sqrt{\det (g_{ij})}}
    	\]
    	is a globally-defined, smooth positive function on $M$; furthermore,  if $M$ has non-empty boundary, then $\mu (g_0) /\mu(g)$ extends continuously up to $\p M$.

    	We claim that $H$ descends to a continuous map from $[0, 1] \times \mathfrak{C}(M)$ to $\mathfrak{C}(M)$. Indeed,  if $\tilde{g} = e^{2u} g$ then
    	\[
    	\det (\tilde{g}_{ij}) = e^{2n u} \det (g_{ij}); \quad \textup{thus}, \quad \frac{\mu(\tilde{g})}{\mu(g_0)} = e^{nu} \frac{\mu(g)}{\mu(g_0)}.
    	\]
    	It follows that
    	\[
    	\bigg(\frac{\mu(g_0)}{\mu(\tilde{g})}\bigg)^{2/n} \tilde{g} = \bigg(\frac{\mu(g_0)}{\mu(g)}\bigg)^{2/n} g,
    	\]
    	hence $H(t, \tilde{g}) = H(t, g)$. The continuity of $H$ on $\mathfrak{C}(M)$ follows immediately from the definition of quotient topology; since $H(0, [g]) = [g]$ and $H(1, [g]) = [g_0]$ we conclude that $\mathfrak{C}(M)$ is homotopy equivalent to $[g_0]$.
    \end{proof}

\bibliographystyle{amsbook}

\end{spacing}

\end{document}